\documentclass[journal]{IEEEtran}

\usepackage{graphicx}
\usepackage[cmex10]{amsmath}
\usepackage{amsfonts}
\usepackage{amssymb}
\usepackage{latexsym}
\usepackage{mathrsfs}
\usepackage{fixltx2e}
\usepackage{epstopdf}
\usepackage{amsmath}

\newtheorem{theorem}{Theorem}

\newtheorem{remark}{Remark}

\begin{document}
\title{On Kapteyn--Kummer Series' Integral Form}

\author{Tibor~K.~Pog\'any, \'Arp\'ad~Baricz,  Anik\'o~Szak\'al \vspace{2mm} \\
{\it Dedicated to Prof. RNDr. Zbyn\v ek N\'aden\'{\i}k to his 90th birth annyversary}

\thanks{{\bf Tibor K. Pog\'any} is with Institute of Applied Mathematics, \'Obuda University, 1034 Budapest, Hungary and Faculty 
of  Maritime Studies, University of Rijeka, 51000 Rijeka, Croatia; {\bf \'Arp\'ad Baricz} is with Institute of Applied Mathematics, 
\'Obuda  University, 1034 Budapest, Hungary and Department of Economics, Babe\c{s}-Bolyai University, 400591 Cluj--Napoca, Romania;  
{\bf Anik\'o Szak\'al} is with \'Obuda University, 1034 Budapest, Hungary.} 
\thanks{{\bf e--mail:}~{poganj@pfri.hr, tkpogany@gmail.com} (T.K. Pog\'any),~{bariczocsi@yahoo.com}
(\'A. Baricz),~{szakal@uni-obuda.hu} (A. Szak\'al).}}

\markboth{}%
{}

\maketitle
\allowdisplaybreaks

\begin{abstract}
In this short research note we obtain double definite integral expressions for the Kapteyn type series built by Kummer's $M$  
(or confluent hypergeometric ${}_1F_1$) functions. These kind of series unify in natural way the similar fashion results for Neumann--, 
Schl\"omilch-- and Kapteyn--Bessel series recently established by Pog\'any, S\"uli, Baricz and Jankov Ma\v sirevi\'c. 
\end{abstract}


\section{Introduction and Preliminaries}

\IEEEPARstart{T}{\rm he} series of Bessel (or Struve) functions in which summation is realized with respect to the indices 
appearing in the order of the building term functions and/or wrapped arguments of the same input functions, can be unified
in a double lacunary form:
   \begin{equation} \label{A0}
	    \mathfrak B_{\ell_1, \ell_2}(z) := \sum_{n \geq 0} \alpha_n \mathscr B_{\ell_1(n)}(\ell_2(n)z).
	 \end{equation}
Here $ x \mapsto \ell_j(x) = \mu_j + a_jx,$ $j\in\{1,2\},$ $x \in\{0,1,\dots\}$, $z \in \mathbb C$ and $\mathscr B_{\nu}$ can be 
chosen from one of Bessel, Struve, Dini and another related special functions and/or their products, \cite{TKP1, JD}. This extension 
of the classical theory of the {\it so--called} Fourier--Bessel series of the first type is based on the case when 
$\mathscr B_{\nu} = J_{\nu}$ for which the thorough account was given in famous Watson's monograph \cite{watson} with extensive 
references list therein. However, specifying varying the coefficients of $\ell_1$ and $\ell_2,$ we appear to three cases related not 
only to physical models and have physical interpretations in many branches of science, technics and technology (consult for instance 
the corner-stone paper by Pog\'any and S\"uli \cite{PS1} and \cite[Introduction]{BP}), but are also of mathematical interest, 
like e.g. zero function series \cite{watson}. Thus, we differ the Neumann series ($a_1 \neq 0, a_2 = 0$) \cite{PS1, BJP, JPS}, 
Schl\"omilch series ($a_1 = 0, a_2 \neq 0$) \cite{JP} and the most general Kapteyn series ($a_1 \cdot a_2 \neq 0$) introduced by Willem 
Kapteyn in \cite{WK, WK1}. 

As our main goal concerns the Kapteyn series we will focus our exposition to this kind of series, pointing out that a set 
of problems associated with Kapteyn type series are solved in \cite{BJP0, JP1}. 

The Kummer's differential equation \cite[\S 13.2]{NIST}
   \[ z\, \frac{{\rm d}^2w}{{\rm d}z^2} + (b-z)\, \frac{{\rm d}w}{{\rm d}z} - aw = 0, \qquad w \equiv M(a, b, z)\]
is the limiting form of the hypergeometric differential equation with the {\it first standard} series solution  
   \[ M(a, b, z) = \sum_{n \geq 0} \frac{(a)_n}{(b)_n}\, \frac{z^n}{n!}, \qquad a \in \mathbb C,\, 
	                 b \in \mathbb C \setminus \mathbb Z_0^-\,.\]
The series converges for all $z \in \mathbb C$. Here $(a)_n = a (a+1) \cdots (a+n-1)$ stands for the standard Pochhammer symbol. 
Another notations which occur for Kummer's function: $\Phi(a;b;z), {}_1F_1(a; b; z)$.

Having in mind the structure of Fourier--Bessel series \eqref{A0} let us introduce the {\it Kapteyn--Kummer series} as
   \begin{align} \label{A1}
	    \mathscr K_\kappa(z) &=: \mathscr K_\kappa \Big( \begin{array}{c} a, b \\ \alpha, \beta, \zeta \end{array} ;z\Big) \nonumber \\
			                     &= \sum_{n \geq 0} \kappa_n\, M(a+\alpha n, b+\beta n, z(1+\zeta n)) \,,
	 \end{align}
where $\kappa_n \in \mathbb C$; the parameter range and the $z$--domain will be described in the sequel. We point out that for at 
least one non--zero $\alpha, \beta$, and $\zeta = 0$, this series becomes a Neumann--, while in the case 
$\alpha = \beta = 0, \zeta \neq 0$ we are faced with the Schl\"omilch--Kummer series. 

We are motivated by the fact that Kummer's function $M(a, b, z)$ generate diverse special functions such as \cite[pp. 507-8, \S 13.6. Special Cases]{AS}
   \begin{align*} 
	    M(\nu+\tfrac12, 2\nu+1, 2{\rm i}z)  &= \Gamma(1+\nu)\, {\rm e}^{{\rm i}z} \big( \tfrac12 z\big)^{-\nu}\!\! J_\nu(z) \\ 
			M(-\nu+\tfrac12, -2\nu+1, 2{\rm i}z) &= \Gamma(1-\nu)\, {\rm e}^{{\rm i}z} \big( \tfrac12 z\big)^\nu \\ \
			     &\hspace{-1.5cm} \times \big[ \cos(\nu\pi) J_\nu(z) - \sin(\nu\pi) Y_\nu(z)\big] \\ 
			M(\nu+\tfrac12, 2\nu+1, 2z) &= \Gamma(1+\nu)\, {\rm e}^z \big( \tfrac12 z\big)^{-\nu} I_\nu(z) \\ 
			M(\nu+\tfrac12, 2\nu+1, 2z) &= \pi^{-\frac12}\, {\rm e}^z (2z)^{-\nu} K_\nu(z)  \,, 
	 \end{align*}
where $J_\nu (I_\nu), Y_\nu (K_\nu)$ stand for the Bessel (modified Bessel) functions of the first and second kind of the order $\nu$ 
respectively, for which their Fourier--Bessel series have been studied in \cite{TKP1, JD, PS1, BJP, JPS, JP} and \cite{JP1},  
among others. Further special cases of $M$ listed in \cite[pp. 507-8, \S 13.6.]{AS} are: Hankel, spherical Bessel, Coulomb wave, 
Laguerre, incomplete gamma, Poisson--Charlier, Weber, Hermite, Airy, Kelvin, error function and also  
elementary functions like trigonometric, exponential and hyperbolic ones. These links from Kummer's $M$ to above mentioned 
special functions and then {\it a fortiori} to their Schl\"omilch--, Neumann-- and Kapteyn--series obviously justify the definition  
of the Kapteyn--Kummer $\mathscr K_\kappa$--series \eqref{A1}. 

Our main aim here is to establish integral representation formula for the Kapteyn--Kummer series $\mathscr K_\kappa$. The main 
derivation tools will be the associated Dirichlet series, the famous Cahen formula \cite{Cah} and the Euler--Maclaurin summation 
formula firstly used in similar purposes in \cite{TP1} and in \cite{PS1}. 

\section{Main results} 
The derivation of the integral representation formula we split into few crucial steps assuming that all auxiliary parameters  
$a, b, \alpha, \beta$ {\it mutatis mutandis} are non-negative, and $\zeta$ real. Further necessary constraints between them 
follow in step--by--step exposition. \medskip

\noindent {\bf 1. The convergence issue.} Having in mind the integral expression of Kummer's function \cite[p. 505, Eq. 13.2.1]{AS}
   \begin{equation} \label{B0}
	    M(a, b, z) = \frac{\Gamma(b)}{\Gamma(b-a)\Gamma(a)} \int_0^1 {\rm e}^{zt}\, t^{a-1}(1-t)^{b-a-1}\, {\rm d}t, 
	 \end{equation}
valid for all $\Re (b) > \Re (a) >0$, we transform the Kapteyn--Kummer series into 
   \begin{align} \label{B01} 
	    &\mathscr K_\kappa(z) = \sum_{n \geq 0} \frac{\kappa_n \Gamma(b+\beta n)}{\Gamma(b-a+(\beta-\alpha)n)\Gamma(a+\alpha n)} \nonumber \\
			          &\quad \times \int_0^1 {\rm e}^{z(1+\zeta n)t}\, t^{a+\alpha n-1}(1-t)^{b-a+(\beta-\alpha)n-1}\, {\rm d}t .
	 \end{align}
Hence, for all $\beta \geq \alpha \geq 0$ using \eqref{B01} we yield
   \begin{align} \label{B02} 
	    &\big|\mathscr K_\kappa(z)\big| \leq \sum_{n \geq 0}\frac{|\kappa_n| \Gamma(b+\beta n)}
			                        {\Gamma(b-a+(\beta-\alpha)n)\Gamma(a+\alpha n)} \nonumber \\
			          &\qquad \times \int_0^1 {\rm e}^{\Re(z)(1+\zeta n)t}\, t^{a+\alpha n-1}(1-t)^{b-a+(\beta-\alpha)n-1}\, {\rm d}t \notag \\
								&\leq  \sum_{n \geq 0}\frac{|\kappa_n| \Gamma(b+\beta n)}
			                        {\Gamma(b-a+(\beta-\alpha)n)\Gamma(a+\alpha n)} \nonumber \\
			          &\quad \times \int_0^1 {\rm e}^{|\Re(z)|(1+|\zeta| n)t}\, t^{a+\alpha n-1}(1-t)^{b-a+(\beta-\alpha)n-1}\, {\rm d}t \notag \\
								&\leq  {\rm e}^{|\Re(z)|}\sum_{n \geq 0}\frac{|\kappa_n| \Gamma(b+\beta n)\,{\rm e}^{|\zeta \Re(z)|n}}
			                        {\Gamma(b-a+(\beta-\alpha)n)\Gamma(a+\alpha n)} \nonumber \\
			          &\qquad \times \int_0^1 \, t^{a+\alpha n-1}(1-t)^{b-a+(\beta-\alpha)n-1}\, {\rm d}t \notag \\
								&=     {\rm e}^{|\Re(z)|} \sum_{n \geq 0} |\kappa_n|\, {\rm e}^{|\zeta \Re(z)|n}\,.
	 \end{align}
Here we employ the Euler Beta function's integral form and its connection to the Gamma function: 
   \[ {\rm B}(p, q) = \int_0^1 t^{p-1} (1-t)^{q-1}\, {\rm d}t = \frac{\Gamma(p)\,\Gamma(q)}{\Gamma(p+q)},\]
where $\min\big(\Re(p), \Re(q)\big)>0$. Indeed, specifying $p = a+ \alpha n$, $q = b-a + (\beta-\alpha)n$ \eqref{B02} 
immediately follows. Finally, by virtue of e.g. Cauchy's convergence test we get the convergence region of $\mathscr K_\kappa(z)$:
   \[ \mathsf R_\kappa'(\zeta) = \left\{ z \in \mathbb C \colon |\zeta \Re(z)| < 
			                  - \log \lim_{n \to \infty} \sqrt[n]{|\kappa_n|} \right\}\,, \]
for any fixed real $\zeta$.  \medskip

\noindent {\bf 2. The associated Dirichlet series.} The Dirichlet series
   \[ \mathscr D_{\boldsymbol a}(r) = \sum_{n \ge 1} a_n\, {\rm e}^{-r \lambda_n}, \]
where $\Re(r)>0,$ having positive monotone increasing divergent to infinity sequence $(\lambda_n)$, possesses Cahen's Laplace 
integral representation formula \cite[p. 97]{Cah}
   \begin{align*} 
      \mathscr D_{\boldsymbol a}(r) &= r \int_0^\infty {\rm e}^{-r t} \sum_{n \colon \lambda_n \leq t} a_n\, {\rm d}t \\ 
							 &= r \int_0^\infty \int_0^{[\lambda^{-1}(t)]} \mathfrak d_u a(u)\, {\rm d}t\,{\rm d}u\, ,
   \end{align*}
where $\mathfrak d_x = 1+\{x\}\frac{\rm d}{{\rm{d}}x}$ and $a \in {\rm C}^1(\mathbb R_+); (a_n) = a\big|_{\mathbb N}$, consult 
\cite{TP1, PS1}.\footnote{Here, $[x]$ and $\{x\}=x-[x]$ denote the integer and fractional part of $x \in \mathbb{R}$, 
respectively.} Indeed, the so--called counting sum 
   \[ \mathscr A_{\boldsymbol a}(t) = \sum_{n \colon \lambda_n \leq t} a_n\] 
we calculate by the Euler--Maclaurin summation formula, see \cite{TP1, PS1}. Hence, 
   \[ \mathscr A_{\boldsymbol a}(t) = \sum_{n=1}^{[\lambda^{-1}(t)]} a_n
                                    = \int_0^{[\lambda^{-1}(t)]} \mathfrak d_u a(u)\, {\rm d}u  \, ,\]
as $\lambda \colon \mathbb R_+ \mapsto \mathbb R_+$ is monotone, there exists unique inverse $\lambda^{-1}$ for the
function $\lambda \colon \mathbb R_+ \mapsto \mathbb R_+$, being $\lambda|_{\mathbb N} = (\lambda_n)$. 

The integral representation formula \eqref{B0} of Kummer's function enables to 
re--formulate the series \eqref{B01} into the following form
   \begin{align} \label{B2} 
	    \mathscr K_\kappa(z) &= \sum_{n \geq 0} \frac{\kappa_n \Gamma(b+\beta n)}{\Gamma(b-a+(\beta-\alpha)n)\Gamma(a+\alpha n)} \nonumber \\
			          &\qquad \times \int_0^1 {\rm e}^{z(1+\zeta n)t}\, t^{a+\alpha n-1}(1-t)^{b-a+(\beta-\alpha)n-1}\, {\rm d}t \notag \\
								&= \int_0^1 {\rm e}^{zt}\, t^{a-1}(1-t)^{b-a-1} \, \mathscr D_\kappa(t) \, {\rm d}t  \,,
	 \end{align}
where the Dirichlet series 
   \[ \mathscr D_\kappa(t) = \sum_{n \geq 0} \frac{\kappa_n\, \Gamma(b+\beta n)\,
	                           {\rm e}^{-\mathfrak p_t n}}{\Gamma(b-a+(\beta-\alpha)n)\Gamma(a+\alpha n)} \, .\]
Here the parameter $\mathfrak p_t = \log \left( t^{-\alpha} (1-t)^{\alpha-\beta}\right)-z\zeta t$ should have 
positive real part. In turn, bearing in mind that for $\zeta \Re(z)<0$ for all $t \in (0,1)$ it is
   \[  \Re(\mathfrak p_t) = -\alpha \log t-(\beta-\alpha)\log (1-t)- \zeta \Re(z)  t >0\, ,\]
we have to take into account the following subset of $\mathsf R_\kappa'(\zeta)$:
   \[ \mathsf R_\kappa(\zeta) = \left\{ z \in \mathbb C : \log \lim_{n \to \infty} \sqrt[n]{|\kappa_n|}<\zeta \Re(z) < 0 \right\}\,. \]
Using $z \in \mathsf R_\kappa(\zeta)$ being $\zeta$ fixed real, applying Cahen's formula and the consequent Euler--Maclaurin 
summation's condensed writing developed in \cite{TP1}, we arrive at  

\begin{theorem} {\it Let $\kappa \in {\rm C}^1(\mathbb R_+)$ be the function which restriction into $\mathbb N_0$ is the 
sequence $(\kappa_n)$. For all $b>a>0$; $\beta \geq \alpha \geq 0$; $\zeta \in \mathbb R$ and for all $z \in \mathsf R_\kappa(\zeta)$, we 
have
   \begin{equation} \label{B6}
	    \mathscr D_\kappa(t) = \frac{\kappa_0 \Gamma(b)}{\Gamma(b-a)\Gamma(a)}
			                     + \mathfrak p_t \int_0^\infty {\rm e}^{-\mathfrak p_t s} \mathscr A_\kappa(s)\, {\rm d}s\,, 
	 \end{equation}
where $\mathfrak p_t = \log \left( t^{-\alpha} (1-t)^{\alpha-\beta} {\rm e}^{-z\zeta t}\right)$ and}
   \[ \mathscr A_\kappa(s) = \int_0^{[s]} \mathfrak d_u\Big( \frac{\kappa(u)\, \Gamma(b+\beta u)}{\Gamma(b-a+(\beta-\alpha)u)
	                           \Gamma(a+\alpha u)}\Big)\, {\rm d}u\, . \]
\end{theorem} 

\begin{proof} It only remains to explain the sum--structure of \eqref{B6}. As to the use of Cahen formula for the Dirichlet series, which involves 
summation over $n \in \mathbb N$, we re--write
   \[ \mathscr D_\kappa(t) = \frac{\kappa_0 \Gamma(b)}{\Gamma(b\!-\!a)\Gamma(a)} + \sum_{n \geq 1} 
                             \frac{\kappa_n\, \Gamma(b+\beta n)\,{\rm e}^{-\mathfrak p_t n}}
									           {\Gamma(b\!-\!a+(\beta\!-\!\alpha)n)\Gamma(a+\alpha n)} \, .\]
The rest is straightforward. 
\end{proof} 

\begin{remark} Obviously the constituting addend constant term $\kappa_0 \Gamma(b)\,(\Gamma(b-a)\Gamma(a))^{-1}$ can be avoided in 
the Dirichlet series' integral expression \eqref{B6} by considering $\kappa_0 = 0$ without loss of any generality. 
\hfill $\blacksquare$ \medskip
\end{remark}

\noindent {\bf 3. The master integral formula for $\mathscr K_\kappa(z)$.} In this subsection of the section II we will 
need further special functions and auxiliary results. Firstly, we recall the double series definition of the so--called 
Kamp\'e de F\'eriet hypergeometric function 
of two variables \cite{AK} in a notation given by Srivastava and Panda \cite[p. 423, Eq. (26)]{SP}. For this, let 
$(H_h)$ denotes the sequence of parameters $(H_1, \cdots, H_h)$ and for nonnegative integers signify the product of Pochhammer 
symbols $((H_h)) := (H_1)_n(H_2)_n \cdots (H_h)_n$, where when $n=0$, the product is understood to reduce to unity. Therefore, the 
convenient generalization of the Kamp\'e de F\'eriet function is defined as follows:
   \begin{align*} 
      F_{\,g:c;d}^{h:a;b}&\Big[ \begin{array}{c} (H_h) \, \colon \,(A_a)\,;\,(B_b) \\ (G_g)\, \colon \,(C_c)\,;\,(D_d) \end{array} 
			                      \Big|  \begin{array}{c} x\\y \end{array} \Big] \notag \\
                         &= \sum_{m,n \geq 0} \dfrac{((H_h))_{m+n}((A_a))_m((B_b))_n}{((G_g))_{m+n}((C_c))_m((D_d))_n}\, 
                            \dfrac{x^m}{m!}\, \dfrac{y^n}{n!}\, .
   \end{align*}
Putting now the integral expression \eqref{B6} of the Dirichlet series $\mathscr D_\kappa(t)$ into the integral form \eqref{B2} 
of the Kapteyn--Kummer series $\mathscr K_\kappa(z)$, by \eqref{B0}, we deduce  
   \begin{align} \label{X2}
      \mathscr K_\kappa(z) &= \kappa_0\, M(a, b, z) \notag \\
                &\quad + \int_0^1\int_0^\infty {\rm e}^{zt}t^{a-1} (1-t)^{b-a-1} \mathfrak p_t \mathscr A_\kappa(s)\,{\rm d}t{\rm d}s \, .
   \end{align}
Let us concentrate to the double integral $\mathscr I_\kappa(z)$ appearing above. By the legitimate change of integration order we have
   \begin{align} \label{X3}
      \mathscr I_\kappa(z) &= - \int_0^\infty \mathscr A_\kappa(s) \Bigg( \int_0^1 {\rm e}^{z(1+\zeta s)t}  \notag \\ 
                           &\qquad \times t^{a+\alpha s-1}(1-t)^{b-a+(\beta-\alpha)s-1} \notag \\
                           &\qquad \times \big( \zeta z t + \alpha \log t + (\beta-\alpha)\log(1-t) \big)\,{\rm d}t\Bigg) {\rm d}s \notag \\
                           &=: - \int_0^\infty \mathscr A_\kappa(s) \Big( \zeta z \mathscr J_\kappa (z, 1) 
                            + \alpha \frac{\partial}{\partial a} \mathscr J_\kappa (z, 0)  \notag \\
                           &\qquad + \beta \frac{\partial}{\partial b} \mathscr J_\kappa (z,0)\Big) {\rm d}s\,,    
   \end{align}
where for $\rho \in \{0,1\}$ the following auxiliary integral occurs:
   \[ \mathscr J_\kappa (z, \rho) = \int_0^1 {\rm e}^{z(1+\zeta s)t} t^{a+\alpha s-1+\rho}(1\!-\!t)^{b-a+(\beta-\alpha)s-1}{\rm d}s. \]
In turn, by \eqref{B0} it is explicitly 
   \[ \mathscr J_\kappa (z, \rho) = {\boldsymbol \Gamma_\rho(s)}\, M\big(a+\alpha s+\rho, b+\beta s+\rho, z(1+\zeta s)\big)\,, \]
where we use the short--hand
   \[ {\boldsymbol \Gamma_\rho(s)} = \frac{\Gamma(b-a+(\beta-\alpha)s)\Gamma(a+\alpha s+\rho)}{\Gamma(b+\beta s+\rho)} \,. \]
                    
\begin{theorem} {\it Let $\kappa \in {\rm C}^1(\mathbb R_+)$ be the function for which $\kappa\big|_{\mathbb N_0} = (\kappa_n)$. 
For all $b>a>0$; $\beta \geq \alpha>0$; $\zeta \in \mathbb R$ and for all $z \in \mathsf R_\kappa(\zeta)$, we 
have}
   \begin{align} \label{X4}
      &\mathscr K_\kappa(z) = \kappa_0\, M(a, b, z) \notag \\ 
                &\qquad   - \int_0^\infty \int_0^{[s]} \mathfrak d_u\Big( \frac{\kappa(u)\, \Gamma(b+\beta u)}{\Gamma(b-a+(\beta-\alpha)u)
	                          \Gamma(a+\alpha u)}\Big)\notag \\
	              &\quad    \times \Big( \zeta z {\boldsymbol \Gamma_1(s)}\, M\big(a+\alpha s+1, b+\beta s+1, z(1+\zeta s)\big) \notag \\
	              &\qquad   + M^* \big( \beta \frac{\partial}{\partial b} {\boldsymbol \Gamma_0(s)}
								          + \alpha \frac{\partial}{\partial a} {\boldsymbol \Gamma_0(s)}\big) \notag \\
	              &\qquad   + {\boldsymbol \Gamma_0(s)} \big( \beta \frac{\partial M^*}{\partial b}  
								          + \alpha \,\frac{\partial M^*}{\partial a}\big) \Big) {\rm d}s\,{\rm d}u. 
   \end{align} 
{\it where $\mathscr A_\kappa(s)$ and $\Gamma_\rho(s), \rho = 0, 1$ are described previously, while 
$M^* := M\big(a+\alpha s, b+\beta s, z(1+\zeta s)\big)$. Accordingly}
   \begin{align*}
	    \frac{\partial M^*}{\partial a} &= \frac{z(1+\zeta s)}{b+\beta s} \notag \\
			                     &\quad \times F_{\,2:0;1}^{1:1;2}\Big[\!\! \begin{array}{c} a+\alpha s+1: 1; 1, a+\alpha s \\ 
                                         2, b+\beta s+1:-; a+\alpha s+1 \end{array} \!\!\Big| \!\! 
												                 \begin{array}{c} z(1+\zeta s)\\ z(1+\zeta s) \end{array}\!\! \Big] \\
      \frac{\partial M^*}{\partial b} &= -\frac{(a+\alpha s)z(1+\zeta s)}{(b+\beta s)^2} \notag \\
			                     &\quad \times F_{\,2:0;1}^{1:1;2}\Big[ \begin{array}{c} a+\alpha s+1: 1; 1, b+\beta s \\ 
                                         2, b+\beta s+1:- ; b+\beta s+1\end{array} \!\!\Big| \!\! 
																				 \begin{array}{c} z(1+\zeta s)\\ z(1+\zeta s) \end{array}\!\! \Big] \, .
	 \end{align*}
\end{theorem} 

\begin{proof} Collecting all these expressions, that is \eqref{X2} and \eqref{X3}, we finish the proof. So, from 
   \begin{align*} 
      &\mathscr K_\kappa(z) = \kappa_0\, M(a, b, z) \notag \\ 
                &\qquad   - \int_0^\infty \int_0^{[s]} \mathfrak d_u\Big( \frac{\kappa(u)\, \Gamma(b+\beta u)}{\Gamma(b-a+(\beta-\alpha)u)
	                          \Gamma(a+\alpha u)}\Big)\notag \\
	              &\quad    \times \Big( \zeta z {\boldsymbol \Gamma_1(s)}\, M\big(a+\alpha s+1, b+\beta s+1, z(1+\zeta s)\big) \notag \\
	              &\qquad   + \beta \frac{\partial}{\partial b} {\boldsymbol \Gamma_0(s)}\, 
	                          M\big(a+\alpha s, b+\beta s, z(1+\zeta s)\big) \notag \\
	              &\qquad   + \alpha \,\frac{\partial}{\partial a} {\boldsymbol \Gamma_0(s)} 
								            M\big(a+\alpha s, b+\beta s, z(1+\zeta s)\big)\Big) {\rm d}s\,{\rm d}u, 
   \end{align*}
with some algebra the double integral will take the form
   \begin{align*} 
      \int_0^\infty &\int_0^{[s]} \mathfrak d_u\Big( \frac{\kappa(u)\, \Gamma(b+\beta u)}{\Gamma(b-a+(\beta-\alpha)u)
	                          \Gamma(a+\alpha u)}\Big)\notag \\
	              &\quad    \times \Big( \zeta z {\boldsymbol \Gamma_1(s)}\, M\big(a+\alpha s+1, b+\beta s+1, z(1+\zeta s)\big) \notag \\
	              &\qquad   + M^* \Big( \beta \frac{\partial}{\partial b} {\boldsymbol \Gamma_0(s)}
								          + \alpha \frac{\partial}{\partial a} {\boldsymbol \Gamma_0(s)}\Big) \notag \\
	              &\qquad   + {\boldsymbol \Gamma_0(s)} \Big( \beta  \frac{\partial M^*}{\partial b}  
								          + \alpha \,\frac{\partial M^*}{\partial a}\Big) \Big) {\rm d}s\,{\rm d}u. 
   \end{align*}
Applying the formulae \cite{http1, http2}
   \begin{align*} 
	    \frac{\partial}{\partial a} M(a, b, z) &= \frac zb F_{\,2:0;1}^{1:1;2}\Big[ \begin{array}{c} a+1: 1; 1, a \\ 
                                    2, b+1: - ; a+1 \end{array} \Big| \begin{array}{c} z\\ z \end{array} \Big] \\
      \frac{\partial}{\partial b} M(a, b, z) &= -\frac{az}{b^2} F_{\,2:0;1}^{1:1;2}\Big[ \begin{array}{c} a+1: 1; 1, b \\ 
                                    2, b+1: - ; b+1 \end{array} \Big| \begin{array}{c} z\\ z \end{array} \Big] 
	 \end{align*}
getting the partial derivatives of $M^*$, in which should be specified $a \rightarrow a+\alpha s$, $b \rightarrow b + \beta s$ 
and $z \rightarrow z(1+\zeta s)$, we arrive at the assertion of the Theorem 2. \medskip
\end{proof} 
    
\section{Toward to Neumann--Kummer and Schl\"omilch--Kummer series} 

As we have mentioned earlier in limiting case $\mathsf A.\,\, \alpha \to 0$ we get a {\it two--parameter Kapteyn--Kummer series}; 
when either $\mathsf B.\,\, \zeta \to 0$ or $\mathsf C. \,\,\alpha, \, \zeta \to 0$, this imply a {\it Neumann--Kummer series}. 

In the last possible common--sense case $\mathsf D.\,\, \beta \to 0$ we earn a {\it Schl\"omilch--Kummer series} -- all from 
$\mathscr K_\kappa(z)$ under the conditions of Theorem 2. 

We point out that for the sake of simplicity in this section we take vanishing $\kappa_0$.\bigskip

\noindent $\mathsf A.\,\, \alpha \to 0$. Since $\alpha \to 0$ independently of $\beta$, in this case we have a Kapteyn--Kummer series:
   \begin{align*}
	    \mathscr K_\kappa &\Big( \begin{array}{c} a, b \\ 0, \beta, \zeta \end{array} ;z\Big) =  
                  \int_0^\infty \int_0^{[s]} \mathfrak d_u\Big( \frac{-\kappa(u)\, \Gamma(b+\beta u)}{\Gamma(b-a+\beta u)}\Big)\notag \\
	              &\quad \times \Bigg( \zeta z a {\boldsymbol \Gamma_1(s)}\, M\big(a+1, b+\beta s+1, z(1+\zeta s)\big) \notag \\
	              &\qquad   + \beta \Big(M^*\big|_{\alpha=0} \frac{\partial}{\partial b} {\boldsymbol \Gamma_0(s)}
	                        + {\boldsymbol \Gamma_0(s)} \frac{\partial M^*|_{\alpha=0} }{\partial b} \Big)\Bigg) {\rm d}s\,{\rm d}u. 
													\bigskip
   \end{align*} 
	    
\noindent $\mathsf B.\,\, \zeta \to 0$. This case results in a two--parameter Neumann--Kummer series 
   \begin{align*}
	    \mathscr K_\kappa &\Big( \!\!\begin{array}{c} a, b \\ \alpha, \beta, 0 \end{array} \! ;z\!\Big) = 
			            \! \int_0^\infty \!\!\int_0^{[s]} \!\mathfrak d_u\Big( \frac{-\kappa(u)\Gamma(b+\beta u)/\Gamma(a+\alpha u)}
								            {\Gamma(b-a+(\beta-\alpha)u)} \Big)\notag \\
	              &\qquad   \times \Bigg( M^*|_{\zeta=0} \Big( \beta \frac{\partial}{\partial b}{\boldsymbol \Gamma_0(s)}
								          + \alpha \frac{\partial}{\partial a} {\boldsymbol \Gamma_0(s)}\Big) \notag \\
	              &\qquad \quad   + {\boldsymbol \Gamma_0(s)} \Big(\beta\, \frac{\partial M^*|_{\zeta=0}}{\partial b}  
								          + \alpha \,\frac{\partial M^*|_{\zeta=0}}{\partial a}\Big) \Bigg) {\rm d}s\,{\rm d}u. \bigskip
   \end{align*} 

\noindent $\mathsf C. \,\,\alpha, \, \zeta \to 0$. Further simplification of the previous integral gives one--parameter 
Neumann--Kummer series, reads as follows:
   \begin{align*}
	    \mathscr K_\kappa &\Big( \begin{array}{c} a, b \\ 0, \beta, 0 \end{array} ;z\Big) = 
			                    - \frac\beta{\Gamma(a)} \int_0^\infty \int_0^{[s]} \mathfrak d_u\Big( \frac{\kappa(u)\Gamma(b+\beta u)}
								            {\Gamma(b-a+\beta u)}\Big)\notag \\
	                      &\qquad \times \Bigg( M^*|_{\alpha,\zeta=0} \frac{\partial}{\partial b} {\boldsymbol \Gamma_0(s)}
								          + {\boldsymbol \Gamma_0(s)}  \frac{\partial  M^*|_{\alpha,\zeta=0}}{\partial b}  
								            \Bigg) {\rm d}s\,{\rm d}u. \bigskip
   \end{align*} 
	
\noindent $\mathsf D.\,\, \beta \to 0$. We end this overview of special cases of Master Theorem 2 with the Schl\"omilch--Kummer series 
integral representation formula
   \begin{align*} 
      \mathscr K_\kappa \Big( \begin{array}{c} a, b \\ 0, 0, \zeta \end{array} ;z\Big) &= -\frac{a\zeta z }{b} 
                               \int_0^\infty \int_0^{[s]} \mathfrak d_u \kappa(u)   \\ 
												      &\qquad \times M\big(a+1, b+1, z(1+\zeta s)\big)\,{\rm d}s\,{\rm d}u. 
   \end{align*}

\section{Acknowledgments} 
The authors are indebted to Dragana Jankov Ma\v sirevi\'c and the anonymous referee for providing insightful comments and valuable
suggestions which substantially encompassed the article.
 
The research of \'A. Baricz was supported by the J\'anos Bolyai Research Scholarship of the Hungarian Academy of Sciences.


\begin{thebibliography}{00}
\bibitem{TKP1}
T. K. Pog\'any, {\it Essays on Fourier--Bessel series}, Habilitation thesis, Budapest: Applied Mathematical Institute, 
\'Obuda University, 2015.

\bibitem{JD}
D. Jankov, {\it Integral expressions for series of functions of hypergeometric and Bessel types}, PhD thesis, Zagreb: 
Department of Mathematics, Faculty of Science, University of Zagreb, 2011.

\bibitem{watson}
G. N. Watson, {\it A Treatise on the Theory of Bessel Functions}, Cambridge: Cambridge University Press, 1922. 

\bibitem{PS1}
T. K. Pog\'any and E. S\"uli, ``Integral representation for Neumann series of Bessel functions," {\em Proc. Amer. Math. Soc.} 
{\bf 137}, pp. 2363--2368, 2009. 

\bibitem{BP}
\'A. Baricz and T. K. Pog\'any, ``Integral representations and summations of modified Struve function," {\it Acta Math. Hungar.} 
{\bf 141(3)} pp. 254-281, 2013. 

\bibitem{BJP}
\'A. Baricz, D. Jankov and T. K. Pog\'any, ``Integral representations for Neumann-type series of Bessel functions
$I_\nu, Y_\nu$ and $K_\nu$," {\it Proc. Amer. Math. Soc.} {\bf 140} (2012), pp. 951--960, 2012.

\bibitem{JPS}
D. Jankov, T. K. Pog\'any and E. S\"uli, ``On the coefficients of Neumann series of Bessel functions,"
{\it J. Math. Anal. Appl.} {\bf 380(2)}, pp. 628--631, 2011. Corrigendum to ``On the coefficients of Neumann series of 
Bessel functions": [J. Math. Anal. Appl. 380 (2011), No. 2, 628-631], \footnotesize{http://www.cs.ox.ac.uk/endre.suli/corrigendum.pdf}. 

\bibitem{JP}
D. Jankov and T. K. Pog\'any, ``Integral representation of Schl\"omilch series," {\it J. Classic. Anal.} {\bf 1(1)}, pp. 
75--84, 2012. 

\bibitem{WK}
W. Kapteyn, ``Recherches sur les functions de Fourier--Bessel," {\em Ann. Sci. de l'\'Ecole Norm. Sup.} {\bf 10}, pp. 91--120, 1893.

\bibitem{WK1}
W. Kapteyn, ``On an expansion of an arbitrary function in a series of Bessel functions," {\em Messenger of Math.} {\bf 35}, 
pp. 122--125, 1906.

\bibitem{BJP0}
\'A. Baricz, D. Jankov and T. K. Pog\'any, ``Integral representation of first kind Kapteyn series," 
{\it J. Math. Phys.} {\bf 52}, Art. 043518, 7pp, 2011.

\bibitem{JP1}
D. Jankov and T. K. Pog\'any, ``On coefficients of Kapteyn-type series," {\it Math. Slovaca} {\bf 65(2)}, pp. 403--410, 2014.

\bibitem{NIST}
F. W. J. Olver, D. W. Lozier, R. F. Boisvert and C. W. Clark (Editors), {\it NIST Handbook of Mathematical Functions}, 
United States Department of Commerce: Cambridge University Press, 2010.

\bibitem{AS}
M. Abramowitz, and I. A. Stegun (Editors), {\it Handbook of Mathematical Functions with Formulas, Graphs, and Mathematical Tables}. 
Applied Mathematics Series {\bf 55} (10 ed.), New York: United States Department of Commerce, National Bureau of Standards; 
Dover Publications, 1964. 

\bibitem{Cah}
E. Cahen, ``Sur la fonction $\zeta(s)$ de Riemann et sur des fontions analogues,"  
{\it Ann. Sci. l'\'Ecole Norm. Sup. S\'er. Math.} {\bf 11}, pp. 75--164, 1894.

\bibitem{TP1}
T. K. Pog\'any, ``Integral representation of a series which includes the Mathieu $\boldsymbol a$--series,"  
{\it J. Math. Anal. Appl.} {\bf 296(1)}, pp. 309--313, 2004. 

\bibitem{AK} 
P. Appell and J. Kamp\'e de F\'eriet, {\it Fonctions hypergeometrique. Polynomes d'Hermite}, Paris: Gautier--Villars, Paris, 1926. 

\bibitem{SP} 
H. M. Srivastava and R. Panda, ``An integral representation for the product of two Jacobi polynomials," 
{\it J. London Math. Soc. (2)} {\bf 12}, pp. 419--425, 1976. 

\bibitem{http1}
\footnotesize{functions.wolfram.com/HypergeometricFunctions/Hypergeometric1F1/

20/01/01/0002} 

\bibitem{http2}
\footnotesize{functions.wolfram.com/HypergeometricFunctions/Hypergeometric1F1/

20/01/02/0002}
\end{thebibliography}
\end{document}